\documentclass{amsart}
\usepackage[utf8]{inputenc}
\usepackage{fourier}
\usepackage{graphicx}
\usepackage{physics}
\usepackage{color, amsmath, amssymb}
\usepackage{dsfont}
\usepackage{listings}
\usepackage{tikz-cd}
\usetikzlibrary{trees}

\usepackage{amsthm}
\usepackage{matlab-prettifier}
\usepackage{comment}
\usepackage[margin=1.2in]{geometry}
\usepackage{hyperref}
\hypersetup{
    colorlinks=true,
    linkcolor=blue,
    filecolor=magenta,      
    urlcolor=cyan,
    pdftitle={Overleaf Example},
    pdfpagemode=FullScreen,
    }

\usepackage{algorithm}
\usepackage{algorithmic}

\newcommand{\cS}{\mathcal{S}}
\newcommand{\f}{\phi}
\newcommand{\ra}{\rightarrow}
\newcommand{\sg}{\sigma}

\DeclareMathOperator{\Int}{Int}

\newcommand{\N}{\mathbb N}
\newcommand{\R}{\mathbb R}

\newcommand{\C}{\mathbb C}

\newcommand{\D}{\mathbb D}

\newcommand{\fs}{\mathcal{S}}

\newcommand{\diam}{\text{diam}}

\newcommand{\cH}{\mathcal{H}}

\newcommand{\om}{\omega}

\numberwithin{equation}{section}
\newtheorem{thm}{Theorem}[section]
\newtheorem{theorem}{Theorem}[section]

\newtheorem{cor}[thm]{Corollary}
\theoremstyle{definition}

\newtheorem{prop}[thm]{Proposition}
\theoremstyle{definition}

\newtheorem{defn}[thm]{Definition}
\theoremstyle{definition}

\theoremstyle{definition}

\newtheorem{remark}[thm]{Remark}

\title{The dimension spectrum of the infinitely generated Apollonian gasket}
\author{Vasileios Chousionis}
\address{Department of Mathematics, University of Connecticut}
\email{vasileios.chousionis@uconn.edu}

\author{Dmitriy Leykekhman}
\address{Department of Mathematics, University of Connecticut}
\email{dmitriy.leykekhman@uconn.edu}

\author{Mariusz Urba\'nski}
\address{Department of Mathematics, University of North TexaS}
\email{mariusz.urbanski@unt.edu}

\author{Erik Wendt}
\address{Department of Mathematics, University of Connecticut}
\email{erik.wendt@uconn.edu}

\thanks{V.~C.\ was supported by Simons Foundation Collaboration grant 521845 and by  NSF grant 2247117. M.U. was supported by Simons Foundation Collaboration grant 581668.}
\date{\today}

\begin{document}
\begin{abstract} We prove that the infinitely generated Apollonian gasket has full Hausdorff dimension spectrum. Our proof, which is computer assisted, relies on an iterative technique introduced by the first three authors in \cite{dimspec} and on a flexible method for rigorously estimating Hausdorff dimensions of limit sets of conformal iterated function systems, which we recently developed in \cite{chousionis2024rigoroushausdorffdimensionestimates}. Another key ingredient in our proof is obtaining reasonably sized distortion constants for the (infinite) Apollonian iterated function system. 
\end{abstract}

\maketitle
\section{Introduction} The Apollonian (or curvilinear Sierpinski) gasket is one of the most well known and widely studied fractals. Beyond its prevalence in fractal geometry and dynamical systems it has important connections to several branches of mathematics such as number theory, hyperbolic geometry and Fuchsian groups, and group theory, see for example the surveys \cite{oh,pollicott}.

The gasket is the residual set of the standard Apollonian circle packing. Starting with three mutually tangent circles with no triple tangencies, we let $T_1$ be the residual curvilinear triangle. A famous theorem of Apollonius of Perga asserts that there are two circles (known as the \textit{inner and outer Soddy circles}) which are tangent to the three original circles. The inner Soddy circle lies in the interstice between the mutually tangent original circles and gives rise to three new residual curvilinear triangles inside $T_1$ whose union is denoted by $T_2$. Continuing iteratively, three new inner Soddy circles can be found inside $T_2$ giving rise to $9$ residual curvilinear triangles inside $T_2$ and so on. The set $\cap_{n=1}^\infty T_n$ is called an Apollonian gasket. Notice that, since the Hausdorff dimension is a conformal invariant (as it is bi-Lipschitz invariant), the Hausdorff dimension of the residual set (the gasket) is independent of the size of the first three circles.

The dimension theory of the Apollonian gasket has attracted considerable attention from various researchers since the 1960s \cite{baifinch,boyd,chousionis2024rigoroushausdorffdimensionestimates,hirst,mcmullen,thomasdar, Vytnova_Wormell_2025}. 
However, several of these contributions were not rigorously justified and until 2024 the best rigorous estimates for the Hausdorff dimension of the Apollonian gasket were due to Boyd \cite{boyd} from the early 1970s. Recently, Vytnova and Wormell \cite{Vytnova_Wormell_2025} obtained rigorous Hausdorff dimension estimates with precision of 128 digits. In \cite{chousionis2024rigoroushausdorffdimensionestimates}, we introduced a general framework for rigorously estimating the Hausdorff dimension of a wide class of conformal fractals—specifically, the limit sets of maximal conformal graph-directed Markov systems—including the Apollonian gasket and its various subsystems. Although our approach yields estimates that are considerably less precise than those in \cite{Vytnova_Wormell_2025}, it is very fast and offers remarkable flexibility for estimating Hausdorff dimensions of subsystems with 4-5 digits of precision. As will be clear from the following presentation, this feature is absolutely essential for the proof of Theorem \ref{mainintro}. 

It is well-known that the Apollonian gasket can be obtained as a limit set of an IFS consisting of three conformal maps. However, these generators are not contracting (only non-expansive), and this is a serious obstacle in analyzing the dimensional and geometric properties of the Apollonian gasket. The recent advances \cite{chousionis2024rigoroushausdorffdimensionestimates, Vytnova_Wormell_2025} rely on the fact that the Apollonian gasket can also be viewed as the limit set of an infinitely generated conformal iterated function system (CIFS) consisting of uniformly contracting maps. More precisely, there exists a CIFS $\mathcal{A}$ (see the beginning of Section \ref{sec:bddist}) whose limit set $J_{\mathcal{A}}$ is equal to the Apollonian gasket minus the countable set of its cusp points. This was realized by Mauldin and Urba\'nski in \cite{MUapol} and it was used to verify that the Apollonian gasket has positive and finite $h$-Hausdorff measure and infinite $h$-packing measure, where $h=\dim_{H}(J_{\mathcal{A}})=1.30568\dots$ is the Hausdorff dimension of the Apollonian gasket. 

We contribute to the dimension theory of the Apollonian gasket by determining its (Hausdorff) dimension spectrum. We first recall how the dimension spectrum of IFSs is defined. Let $X \subset \R^n$ be a compact set and let $\cS=\{\f_i:X \ra X\}_{i \in I}$ be a countable collection of uniformly contracting maps. If $A \subset I$ we denote the corresponding subsystem of $\cS$ by $\cS_A=\{\f_i:X \ra X\}_{i \in A}$ and we denote its limit set by $J_A$. The {\em dimension spectrum} of $\cS$ is the set of all possible values of the Hausdorff dimension of the subsystems of $\cS$, i.e.  
$$
DS(\cS):=\{\dim_{H}(J_{A}): A \subset I\}.
$$
If $\cS$ consists of finitely many maps, then $DS(\cS)$ is a finite set. However, in the case when $\cS$ is infinite, the structure of the dimension spectrum becomes significantly richer and more intricate.

The dimension spectrum was first studied by Mauldin and Urba\'nski in \cite{MUCIFS} where they showed that if $\cS$ is a CIFS then $[0,\theta) \subset DS(\cS)$ where $\theta$ is the finiteness parameter of the system, see \eqref{eq:theta}. Naturally, one is interested in systems whose dimension spectrum is an interval; i.e. systems which satisfy $DS(\cS)=[0, \dim_{H}(J_{\cS})]$. In that case, we say that $\cS$ has \textit{full dimension spectrum}. In a significant contribution, Kesseb\"ohmer and Zhu \cite{KZ} proved that the dimension spectrum of the CIFS resulting from the real continued fractions algorithm is full, thus resolving the so-called Texan conjecture. The result in \cite{KZ} can be restated as follows; given any $t \in [0,1]$, there exists $E_t \subset \N$ such that the set of irrational numbers whose continued fraction expansion only contains digits from $E_t$ has Hausdorff dimension $t$.

More recently, the first three authors introduced a rigorous computational approach in \cite{dimspec,ifscont}, drawing partial inspiration from the methods developed in \cite{KZ}, to study the dimension spectrum of various conformal fractals associated with continued fractions. In particular, \cite{ifscont} established that the conformal iterated function systems (CIFS) arising from real continued fractions—when restricted to alphabets such as the set of prime numbers, perfect squares, or any arithmetic progression—possess full dimension spectrum. Additionally, \cite{dimspec} demonstrated that the standard CIFS derived from complex continued fractions also has full dimension spectrum. We note that the dimension spectrum of infinite IFSs can be very different than an interval as it may have many non-degenerate connected components and connected components being singletons. Such examples, can be found in \cite{ifscont, DSspec,KZ}. However, it was proved in  \cite{dimspec} that in the case of conformal iterated function systems the dimension spectrum is always compact and perfect. Interestingly, this is not always the case for affine IFSs, see \cite{jurga}.

In this paper we combine methods from \cite{dimspec} together with our recent computational approach from \cite{chousionis2024rigoroushausdorffdimensionestimates} to prove the following theorem.
\begin{thm} 
\label{mainintro} The dimension spectrum of the infinitely generated Apollonian gasket is full, i.e.
$$DS(\mathcal{A})=[0, \dim_{H}(J_{\mathcal{A}})].$$
\end{thm}

We will now discuss briefly the proof of Theorem \ref{mainintro}. Our starting point is \cite[Proposition 6.17]{dimspec}. Given any CIFS $\mathcal{S}=\{\phi_i:X \to X\}_{i \in I}$, a well ordering $\prec$ on the alphabet $I$ is called {\em natural} (with respect to $\cS$) if
$$a \prec b \implies \| \f_a'\|_\infty \geq \| \f_b'\|_\infty.$$
If $I=\{i_n\}_{n \in \N}$ is an enumeration respecting $\prec$, then we denote the $m$-initial segments by $I(m)=\{i_1,\dots,i_m\}.$
According to \cite[Proposition 6.17]{dimspec} if $K$ is a distortion bound for $\cS$ (see \eqref{bdp}) and there exist $t_1 \leq  t_2 \leq \dim_{H}(J_{\cS})$ and $d \in \N$ such that 
\begin{enumerate} 
\item \label{cond1int}$\dim J_{I_d} \geq t_1$,
\item \label{cond2int} $\sum_{n \geq k+1} \| \f_{i_n}'\|_\infty^{t_2} \geq K^{2t_2} \|\f_{i_k}'\|^{t_2}_\infty$ for all $k \geq d+1$.
\end{enumerate}
Then $[t_1,t_2] \subset DS(\cS)$.

If one wants to apply the previous scheme in order to find intervals inside the spectrum of the Apollonian system $\mathcal{A}$, first has to find a computationally friendly distortion constant $K_{\mathcal{A}}$. We deal with this problem in Section \ref{sec:bddist} and we prove in Proposition \ref{distprop} that 
$$
K_{\mathcal{A}} \leq 5.900319.
$$ 
We consider that our distortion bound is of independent interest and likely will be useful for other related problems. Moreover, motivated by our bound, one can naturally ask what is the best distortion constant for the Apollonian system $\mathcal{A}$. We note that although the proof of Proposition \ref{distprop} is quite technical, it only uses elementary tools.

With a reasonably sized distortion constant at hand,  we use an iterative bootstrapping argument in Section \ref{sec:proof}, similar to that from \cite{dimspec}, to identify a chain of $18$ overlapping intervals contained in $DS(\mathcal{A})$ and containing $[1/2, \dim_{H}(J_{\mathcal{A}}]$. This is enough because by \cite{MUapol, MUCIFS} we know that $[0,1/2) \subset DS (\mathcal{A})$. Here is exactly where our flexible approach from \cite{chousionis2024rigoroushausdorffdimensionestimates}
comes very handy. Using the method developed in \cite{chousionis2024rigoroushausdorffdimensionestimates} we are able to rigorously estimate upper and lower bounds for the Hausdorff dimension of various Apollonian subsystems  (see Tables \ref{tab:updimest}) with sufficient accuracy for our bootstrapping argument to work.
All $18$ steps are summarized in Table  \ref{tab:main}.
We would like to point out that obtaining a good upper bound for the distortion constant of $\mathcal{A}$ is essential; it allows us to take the first crucial step which produces a small interval in the spectrum that includes $\dim_{H} (J_{\mathcal{A}})$.

The paper is structured as follows. Section \ref{sec:prelim} introduces the necessary background for conformal iterated function systems and their thermodynamic formalism. In Section \ref{sec:bddist}, which is the main technical part of our paper, we obtain an estimate for the distortion constant of the Apollonian system. Finally, in Section \ref{sec:proof} we prove Theorem \ref{mainintro}.

\section{Preliminaries}
\label{sec:prelim} We start by introducing some necessary background on conformal iterated function systems. Let $X$ be a compact metric space. An iterated function system (IFS) on $X$ is a countable family $\cS=\{\phi_{i}:X \to X\}_{i \in I}$ of injective contractions whose Lipschitz constant is less than $s$ for some $s \in (0,1)$. 

If $\om \in I^*:=\bigcup_{n=0}^\infty I^n$, we denote by $|\om|$  the unique integer
$n \geq 0$ such that $\om \in I^n$, and we call $|\om|$ the {\em length} of
$\om$. If $\om \in I^\N$ or $|\om|>n$ we let $
\om |_n:=\om_1\ldots \om_n\in I^n.
$ Given $\om \in I^{\mathbb N}$ we denote
$\phi_{\om|_n}=\phi_{\om_1} \circ \dots \phi_{\om_n}.$ Note that the non-empty compact sets
$\{\f_{\om|_n}(X)\}_{n=1}^\infty$ form a decreasing sequence with
nonempty intersection.
Since $\diam(\f_{\om|_n}(X)) \le s^n\diam(X), \forall n \in \N,$ we can define the coding map $\pi:I^{\mathbb N}_A\to X$ given by
$$
\pi(\om):=\bigcap_{n\in  \N}\f_{\om|_n}(X).
$$
The set
$
J=J_\cS:=\pi(I^{\mathbb N})
$
will be called the {\it limit set} \index{limit set} (or {\it attractor}) \index{attractor} of the IFS $\cS$.

We will be interested in conformal IFSs (CIFS) in the complex plane. 
\begin{defn}
\label{cifs}
Let $X \subset \C$ be compact and connected. An IFS $\cS=\{\phi_i:X \to X\}_{i \in I}$ is called {\it conformal}  if the following conditions are satisfied.
\begin{enumerate}
\item \label{cgdmsiii} There exists an open and connected
set $W \supset X$ such that the maps
$\f_i, i \in I,$ extend to a $C^1$ conformal diffeomorphisms of $W$ into $W$.
\item ({\it Open Set Condition} or {\it OSC}). \index{open set condition} For all $i,j\in
I$, $i\ne j$,
$$
\phi_i(\Int(X)) \cap \phi_j(\Int(X))= \emptyset.
$$
\end{enumerate} 
\end{defn}
It follows by \cite{Mauldin_Urbanski_2003, KU} that CIFSs in the complex plane satisfy the \textit{Bounded Distortion Property (BDP)}: there exists some constant $K \geq 1$, depending only on $\cS$, such that
\begin{equation}
\label{bdp}
K^{-1}\le\frac{|\f'_\om(p)|}{|\f'_\om(q)|}\le K
\end{equation}
for every $\om\in I^*$ and every pair of points $p,q\in X$. If $K$ is as in \eqref{bdp} it is called a \textit{distortion constant of $\mathcal{S}$}. Given $\om \in I^\N$ we let
$$\|\phi'_{\om}\|_\infty:=\max_{x \in X}|\phi'_{\om}(x)|.$$ The Leibniz rule and \eqref{bdp} imply
that if $\om \in I^\ast$ and $\om=\tau \upsilon$ for some $\tau, \upsilon \in I^\ast$, then
\begin{equation}\label{quasi-multiplicativity1}
K^{-1} \|\f'_{\tau}\|_\infty \, \|\f'_{\upsilon}\|_\infty \le
\|\f'_\om\|_\infty \le \|\f'_{\tau}\|_\infty \,
\|\f'_{\upsilon}\|_\infty.
\end{equation}

We will now briefly recall some facts from the thermodynamic formalism of CIFSs, see e.g. \cite{MUCIFS, Mauldin_Urbanski_2003, MRU} for more information. Let $\mathcal{S}=\{\f_i\}_{i\in I}$  be a CIFS. For $t\ge 0$ and $n \in \N$ let
\begin{equation}\label{zn}
Z_n(\cS,t):=Z_{n}(t) := \sum_{\om\in I^n} \|\phi'_\om\|^t_\infty.
\end{equation}
Since, by \eqref{quasi-multiplicativity1},
\begin{equation}
\label{zmn}
Z_{m+n}(t)\le Z_m(t)Z_n(t),
\end{equation}
the sequence $\{\log Z_n(t)\}_{n=1}^\infty$ is subadditive. Thus, the limit
$$
P_\mathcal{S}(t):=P(t):=\lim_{n \to  \infty}  \frac{ \log Z_n(t)}{n}=\inf_{n \in \N} \frac{ \log Z_n(t)}{n}
$$ 
exists and it 
is called the {\em topological pressure} of the system $\mathcal{S}$ evaluated at the parameter $t$. We record that $t \mapsto P(t)$  is a decreasing function on $[0,+\infty)$ with $\lim _{t \ra +\infty} P(t)= -\infty$. Moreover, $P$ is convex and continuous on $\overline{\{t \geq 0: P(t)<\infty\}}$, see e.g. \cite[Proposition 19.4.6]{MRU}. We also let 
\begin{equation}
\label{eq:theta}
\theta(\cS):=\theta=\inf\{t \geq 0: P(t) <+\infty\}\quad \mbox { and }\quad h(\cS):=h=\inf\{t \geq 0: P(t)\leq 0\}.
\end{equation}
It is well known, see e.g. \cite[Proposition 19.4.4]{MRU}, that
\begin{equation}
\label{eq:thetaz1}
\theta=\inf\{t \geq 0: Z_1(t) <+\infty\}.
\end{equation}

The parameter $h(\cS)$ is known as \textit{Bowen's parameter} and it coincides with the Hausdorff dimension of the limit set $J_{\cS}$:
\begin{thm}\cite{MUCIFS}
\label{721}
If $\mathcal{S}$ is a CIFS, then
$$
h(\fs)= \dim_H(J_\mathcal{S})
= \sup \{\dim_\cH(J_F):  \, F \subset E \, \mbox{finite} \, \}.
$$
\end{thm}
We note that Theorem \ref{721} has been generalized to various settings, see e.g. \cite[Theorem 4.2.13]{Mauldin_Urbanski_2003} or \cite[Theorem 7.19]{CTU}.

\section{Effective bounds of the Apollonian gasket's distortion}\label{sec:bddist}
From now on we will focus on the Apollonian gasket. We will first introduce an infinite CIFS whose limit set coincides with the Apollonian gasket minus the countable set of its cusp points. 

 We consider the angles
\[ \theta_k = (-1)^k \frac{2 \pi}{3 } \quad\text{ and }\quad \theta_k' = \frac{2 \pi}{3}k \mod 2 \pi, \quad k=1,\dots, 6,\] and the maps
\begin{align*}
    &f(z) = \frac{(\sqrt{3} -1)z +1}{-z + \sqrt{3}+1}, \ R_{\theta_k}, \text{ and } R_{\theta_k'},
\end{align*}
where $R_\theta$ is the standard complex rotation by the angle $\theta$. Let $\D=B(0,1)$ be the unit disk in the complex plane and let
$I=\{1,2, \ldots, 6 \} \times \N$. The Apollonian IFS is then defined as
\begin{equation}
\label{eq:apolsystem}
\mathcal{A}=\{\phi_{k,n}:\D \to \D\}_{(k,n) \in I}
\quad\mbox{ where }\quad
\phi_{k,n} = R_{\theta_k'} \circ f^n \circ R_{\theta_k} \circ f.\end{equation}
The first two iterations of the maps are displayed in Figures \ref{fig: 2D 1iteration} and \ref{fig: 2D 2iteration}.

\begin{figure}[h]
  \hfill
   \begin{minipage}[t]{.45\textwidth}
    \begin{center}
 \includegraphics[scale=.17]{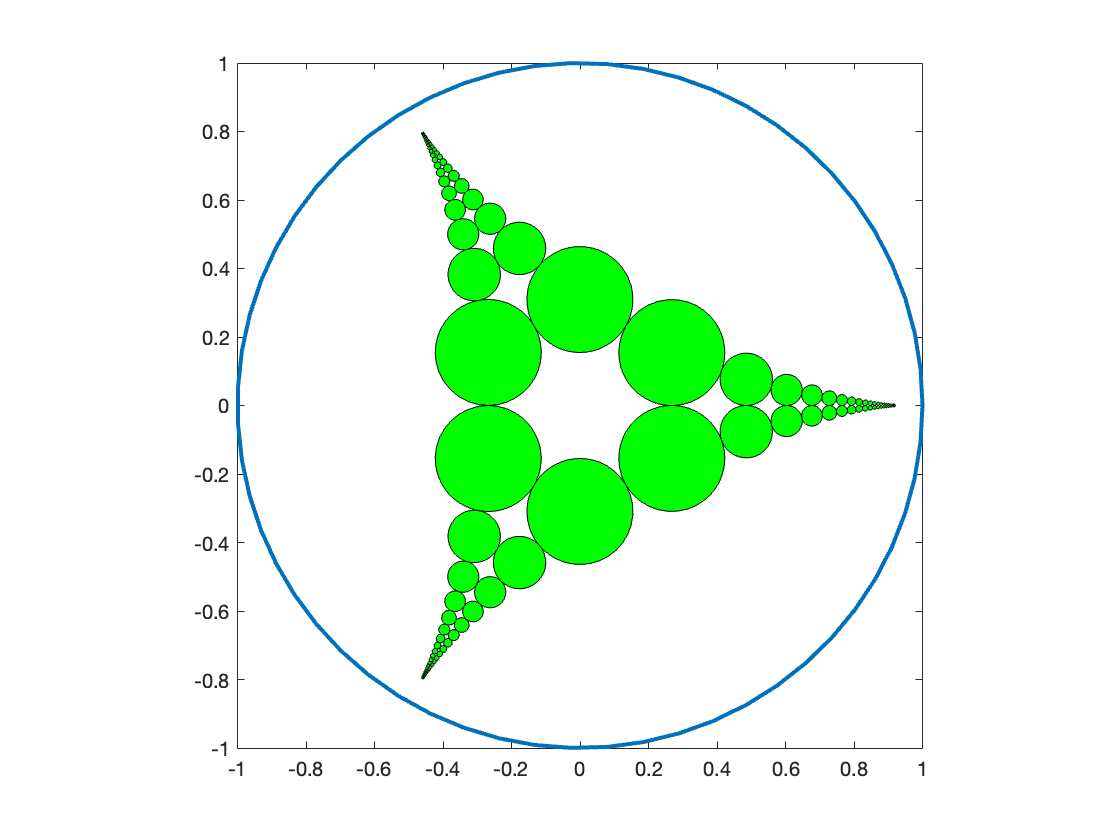}\\
  \caption{1 iteration}
        \label{fig: 2D 1iteration}
    \end{center}
  \end{minipage}
  \hfill
  \begin{minipage}[t]{.45\textwidth}
    \begin{center}
\includegraphics[scale=.17]{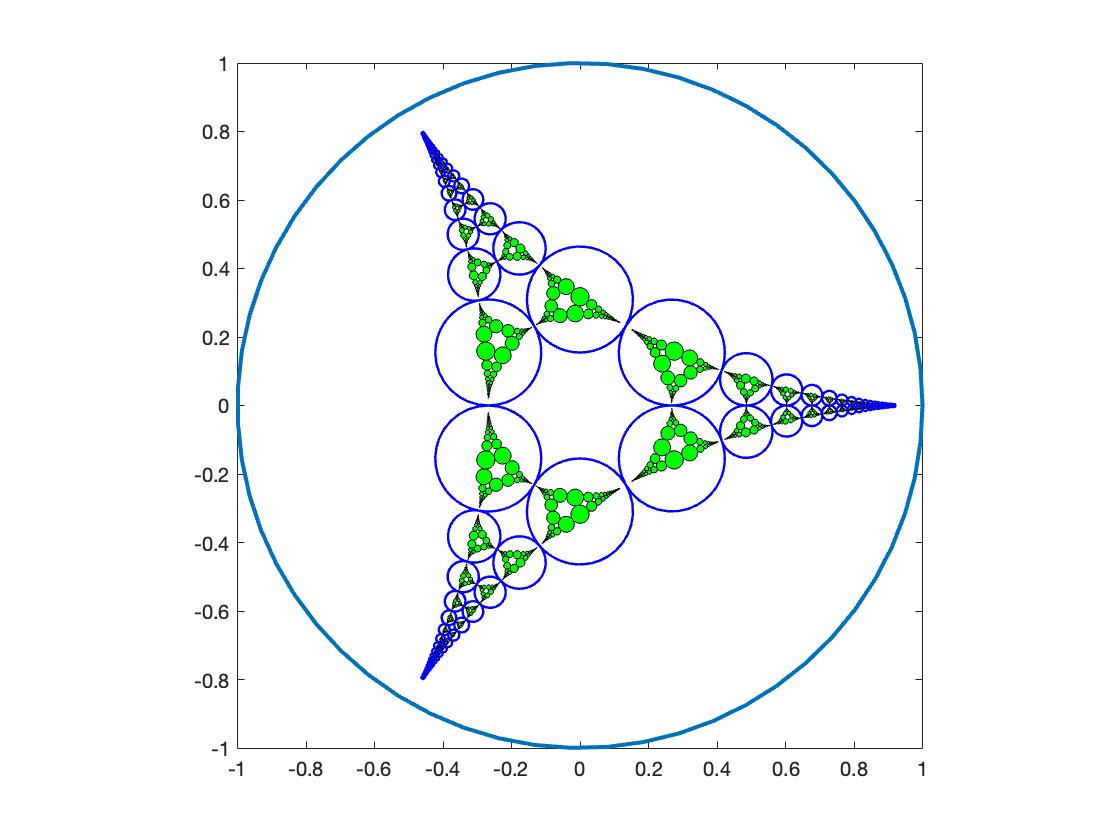}\\
   \caption{2 iterations}
        \label{fig: 2D 2iteration}
          \end{center}
  \end{minipage}
  \hfill
\end{figure}

It was proved in \cite{MUapol} that $\mathcal{A}$ is a CIFS; it satisfies the conditions from Definition \ref{cifs}. We also note that the set $W$ (as it appears in Definition \ref{cifs} \label{cifsw}) it can be taken to be $B(0,1+\sqrt{3})$, i.e.
\begin{equation}
\label{confext}
\phi_{k,n} (B(0,1+\sqrt{3})) \subset B(0,1+\sqrt{3}),
\end{equation}
see \cite[Proposition 6.7]{chousionis2024rigoroushausdorffdimensionestimates}. We will also need to consider certain subsystems of $\mathcal{A}$. If $F\subset \N$ we will let
\begin{equation}
\label{eq:subsystemsofA}
\mathcal{A}_F=\{\phi_{k,n}:\D \to \D\}: (k,n) \in I_F\},\end{equation}
where
$$I_F=\{(k,n): k \in{1,\dots,6}, \quad n \in F\}.
$$

From now on, we will let $\lambda:=\sqrt{3}$. Our proof will depend on effective bounds of the distortion constant of $\mathcal{A}$. These estimates require us to find the radii of the disks $\phi_{k,n} (\D)$ for $n=1,2,3$.
\begin{prop}
\label{prop:disks}
Let $k=1,\dots,6$.
\begin{enumerate}
\item \label{eq:firstdisk} $\mbox{radius}\,\phi_{k,1}(\D)=\frac{2 \lambda-3}{3}$,
\item \label{eq:seconddisk} $\mbox{radius}\, \phi_{k,2}(\D)=\frac{14\lambda-15}{121}$,
\item \label{eq:thirddisk} $\mbox{radius} \, \phi_{k,3}(\D)=\frac{26\lambda - 21}{529}$.
\end{enumerate}
Moreover, the centers of the disks $\phi_{k,1}(\D), k=1,\dots,6,$  lie on $\partial B(0, \frac{4 \lambda-6}{3})$.
\end{prop}
\begin{proof} We will first prove \eqref{eq:firstdisk}. Let $k=1,\dots,6$. The matrix representation of the map  $F_k:=f \circ R_{\theta_k} \circ f$ 
is given by
\begin{equation}\label{eq: mobdecomp}
\begin{pmatrix}
\lambda -1 & 1 \\
-1 & \lambda +1
\end{pmatrix}
\begin{pmatrix}
e^{i\theta_k} &0 \\
0 & 1 
\end{pmatrix}
\begin{pmatrix}
\lambda -1 & 1 \\
-1 & \lambda +1
\end{pmatrix} 
 =\begin{pmatrix}
e^{i\theta_k}(\lambda -1)^2-1 & e^{i\theta_k}(\lambda -1)+(\lambda +1) \\
-e^{i\theta_k}(\lambda -1)-(\lambda +1) & -e^{i\theta_k}+(\lambda +1)^2
\end{pmatrix},
\end{equation}
thus
$$
F_k(z) = \frac{(e^{i\theta_k}(\lambda -1)^2-1)z+e^{i\theta_k}(\lambda -1)+(\lambda +1)}
{(-e^{i\theta_k}(\lambda -1)-(\lambda +1))z-e^{i\theta_k}+(\lambda +1)^2}.
$$
If $k$ is even (i.e. $\theta_k=2\pi/3$,) we set $F_k:=F_e$ and if $k$ is odd (i.e. $\theta=-2\pi/3$) we set $F_k:=F_o$.

We first assume that $k$ is even. Picking 3 points on the unit circle $z_1=1$, $z_2= e^{2\pi/3}$ and $z_3= e^{-2\pi/3}$, we have 
$$
\begin{aligned}
F_e(z_1)&=\frac{(1-i) \left(-\lambda^2+\lambda\, \right)}{\lambda^2i+(1+2i) \lambda} =\frac{1}{2(2+\lambda)}+i \frac{\lambda}{2(2+\lambda)}, \\
F_e(z_2)&=2- \lambda=\frac{1}{2+\lambda}, \\
F_e(z_3)&=\frac{2 \left((-1+i)+\lambda \right)}{(1+2i)+(2+i) \lambda}=\frac{14 }{\left(2+\lambda\, \right)^2+\left(1+2 \lambda\, \right)^2}+ i \frac{2 \lambda\, }{\left(2+\lambda\, \right)^2+\left(1+2 \lambda\, \right)^2}.
\end{aligned}
$$
Thus, we need to compute the center and radius of the circle passing through  
$$
\left(\frac{1}{2(2+\lambda)},\frac{\lambda}{2(2+\lambda)}\right),\quad \left(\frac{1}{2+\lambda},0\right),\quad \left(\frac{14 }{\left(2+\lambda\, \right)^2+\left(1+2 \lambda\, \right)^2},\frac{2\lambda  }{\left(2+\lambda\, \right)^2+\left(1+2 \lambda\, \right)^2}\right).
$$
This is equivalent to solving the $3\times 3$ linear system $Ax=b$ where
\[
A =  \begin{pmatrix}
\frac{1}{2+\lambda} & \frac{\lambda}{2+\lambda} & 1 \\
\frac{2}{2+\lambda} & 0 &  1\\
\frac{28}{\left(2+\lambda \right)^2+\left(1+2 \lambda \right)^2} & \frac{4\lambda}{\left(2+\lambda \right)^2+\left(1+2 \lambda\, \right)^2} &  1\\
\end{pmatrix} 
\quad\mbox{ and
}\quad
b =  \begin{pmatrix}
\frac{2-\lambda}{2+\lambda} \\
\frac{1}{(2+\lambda)^2}\\
\frac{5-2\lambda }{5+2\lambda }
\end{pmatrix} =
\begin{pmatrix}
(2-\lambda)^2 \\
(2-\lambda)^2\\
\frac{(5-2\lambda)^2 }{13 }
\end{pmatrix}.
\]
Inverting $A$ we obtain
$$
A^{-1}=\begin{pmatrix}
\frac{-3-2\lambda}{6} & \frac{-1-\lambda}{2} & \frac{6+5\lambda}{6} \\
\frac{4+3\lambda}{6} & \frac{-9-5\lambda}{6} &  \frac{5+2\lambda}{6}\\
\frac{\lambda}{3} & \lambda  &  \frac{3-4\lambda}{3}\\
\end{pmatrix} = 
\frac{1}{6}\begin{pmatrix}
-3-2\lambda & -3-3\lambda & 6+5\lambda \\
4+3\lambda & -9-5\lambda &  5+2\lambda\\
2\lambda & 6\lambda  &  6-8\lambda
\end{pmatrix}
$$
and consequently:
$$
\begin{pmatrix}
x_1\\ x_2\\x_3
\end{pmatrix}=
A^{-1}b= \frac{1}{6}\begin{pmatrix}
-3-2\lambda & -3-3\lambda & 6+5\lambda \\
4+3\lambda & -9-5\lambda &  5+2\lambda\\
2\lambda & 6\lambda  &  6-8\lambda
\end{pmatrix}
\begin{pmatrix}
(2-\lambda)^2 \\
(2-\lambda)^2\\
\frac{(5-2\lambda)^2 }{13 }
\end{pmatrix}=
\begin{pmatrix}
2-\lambda \\
\frac{2\lambda-3}{3}\\
4\lambda-7
\end{pmatrix}.
$$
Thus,
$$F_e (\D)=B\left(2-\lambda+i\frac{2\lambda-3}{3},\frac{2\lambda-3}{3}\right).$$
Since $F_o(z)=\overline{F_e(\bar{z})}$, we deduce that 
$$F_o(\D)=\overline{F_e(\D)}=B\left(2-\lambda-i\frac{2\lambda-3}{3},\frac{2\lambda-3}{3}\right).$$
Recalling that  $\phi_{k,1}=R_{\theta_k'} \circ f \circ R_{\theta_k} \circ f$, we deduce \eqref{eq:firstdisk}. Moreover, we see that the centers of the circles $\phi_{k,1}(\D),k=1,\dots,6,$ are the points
$$i \frac{4\lambda-6}{3}, -i \frac{4\lambda-6}{3}, 2-\lambda+i\frac{2\lambda-3}{3},2-\lambda-i\frac{2\lambda-3}{3},2-\lambda-i\frac{2\lambda-3}{3}, -2+\lambda-i\frac{2\lambda-3}{3}.$$
These are the six equidistant points on the circle $\partial B\left(0,\frac{4\lambda-6}{3} \right)$. 

It remains to prove \eqref{eq:seconddisk} and \eqref{eq:thirddisk}. Let  
\begin{equation}\label{eq: matrix F}
F = \begin{pmatrix}
\lambda -1 & 1 \\ -1 & \lambda+1
\end{pmatrix},
\end{equation}
be the  matrix representation of the M\"obius map $f(z)=\frac{(\lambda-1)z+1}{-z+(\lambda+1)}$.
The Jordan decomposition of $F$ is given by
\[F=VJV^{-1} = \begin{pmatrix}
-1 & 1 \\
-1 & 0
\end{pmatrix}
\begin{pmatrix}
\lambda & 1 \\ 
0 & \lambda
\end{pmatrix}
\begin{pmatrix}
0 & -1 \\
1 & -1
\end{pmatrix}.\]
Moreover,
\[J^n = \left[ \lambda I + \begin{pmatrix}
0 & 1 \\
0 & 0
\end{pmatrix} \right]^n= \lambda^n I + n \lambda^{n-1}\begin{pmatrix}
0 & 1 \\
0 & 0
\end{pmatrix} = \begin{pmatrix}
\lambda^n & n \lambda^{n-1} \\
0 & \lambda^n
\end{pmatrix}\]
Thus the matrix representation of $f^n$ is
\[F^n = \lambda^n\begin{pmatrix}
-1 & 1 \\
-1 & 0
\end{pmatrix} \begin{pmatrix}
1 & n/ \lambda \\
0 & 1
\end{pmatrix}\begin{pmatrix}
0 & -1 \\
1 & -1
\end{pmatrix} \]
Hence,
$$f^n(z)=\frac{(-n+\lambda)z+n}{-nz+n+\lambda }.$$
We also let 
$$F_{k,n}=f^n \circ R_{\theta_k} \circ f.$$
Thus, we can write:
\begin{equation}
\label{eq:rotationfkn}\phi_{k,n} =R_{\theta_k'} \circ F_{k,n}.
\end{equation}
We will first find $f(\D)$. We note that
$$
f(1)=1 \mbox{ and }
f(-1)= \frac{-(\lambda - 1) +1}{1+\lambda+1} = \frac{2-\lambda}{2+\lambda}=(2-\lambda)^2.
$$
Since $f(z)=-(\lambda -1)-\frac{3}{z-(\lambda+1)}$, (in particular the coefficients of the matrix $F$ are all real numbers) we conclude that the center of $f(\D)$ has to lie in the real axis. Moreover, the center of $f(\D) \notin \{f(1), f(-1)\}$ because otherwise $0 \in f(\D)$ and this is impossible because $f$ only vanishes at $(\lambda-1)^{-1}>1$. Thus, the points $f(1)$ and $f(-1)$ are antipodal in $f(\D)$ and we can conclude that
\begin{equation}
\label{eq:f(D)}
f(\D)=B(2(2-\lambda),\lambda(2-\lambda)).
\end{equation}

Assume that $k$ is even. Note that because $f(1)=1$ we have that $$R_{\theta_k} \circ f (1)=-\frac{1}{2}+i \frac{\sqrt{3}}{2}.$$
Moreover,
$$F_{k,n}(1)=f^n \circ R_{\theta_k} \circ f (1) =\frac{2 n^2-1}{2 n \left(n+\lambda \right)+2}+i\frac{\lambda}{2 n \left(n+\lambda \right)+2}:=b_n.$$
Thus,
\begin{equation}
\label{eq:first blue point}
b_n \in F_{k, n} (\partial \D) \quad \forall n \in \N. 
\end{equation}
However, more is true:
\begin{equation}
\label{eq:sec blue point}
b_{n+1} \in F_{k, n} (\partial \D), \quad \forall n \in \N. 
\end{equation}
One can verify \eqref{eq:sec blue point} by showing that 
$$(f^n)^{-1}(b_{n+1}) \in f(\partial \D).$$
For example,
$$
f^{-1}(b_2)=f^{-1}\left(\frac{7}{10+4\lambda } +i\frac{\lambda}{10+4\lambda }\right)=\frac{2-\lambda}{{2}}+i \frac{2\lambda-3}{2} \in f(\partial \D).
$$
We have thus identified two points on each circle $F_{k,n}(\partial \D)$, see the blue points in Figure \ref{fig: 3 points on circle}.
\begin{figure}[h]
  \hfill
    \begin{center}
 \includegraphics[scale=.25]{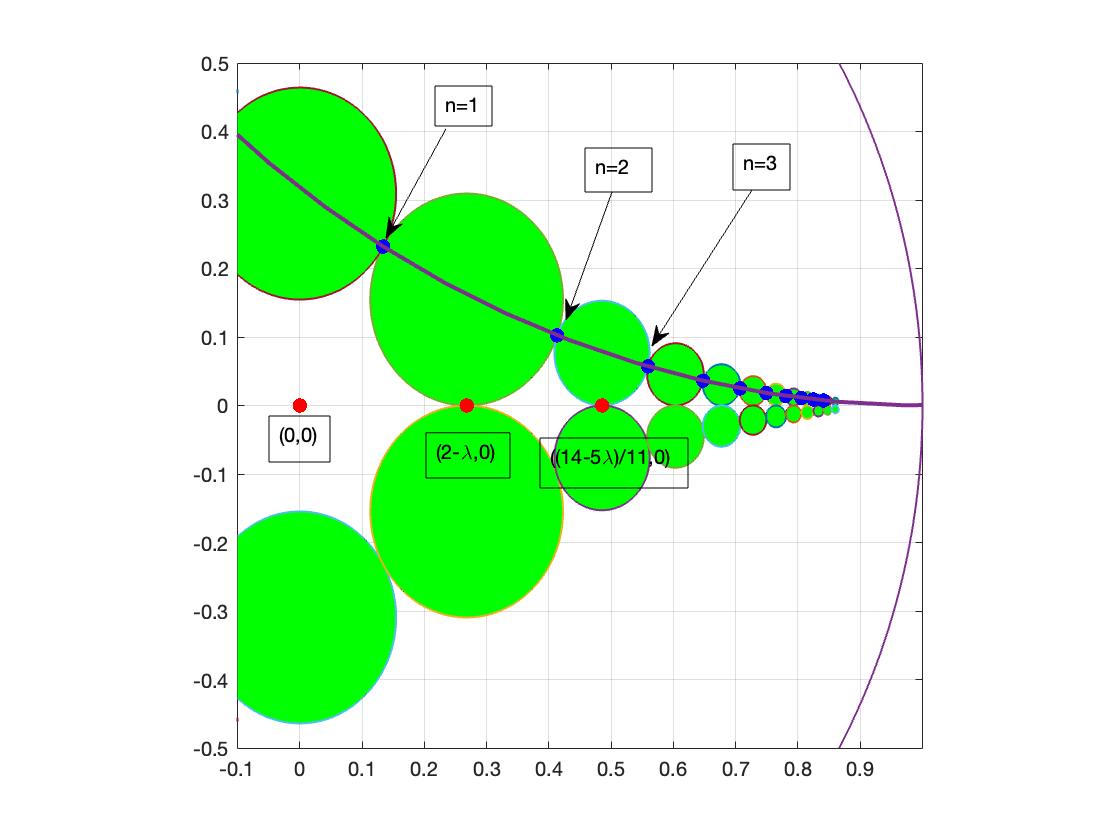}\\
  \caption{Points on the circles}
        \label{fig: 3 points on circle}
    \end{center}
  \hfill
\end{figure}
One can also check  that for all $n=1,2,\dots,$ the points 
$$
\frac{2 n^2-1}{2 n \left(n+\lambda \right)+2}+i \frac{\lambda}{2 n \left(n+\lambda \right)+2}
$$
lie on the circle of radius $R_0=\lambda$ centered at $1+i\lambda$ since
\[\left(\frac{2n^2-1}{2n \left(n+\lambda\, \right)+2}-1\right)^2+\left(\frac{\lambda}{2 n \left(n+\lambda \right)+2}-\lambda \right)^2=3.
\]

We will now describe how to identify a third point in each circle $F_{k,n} (\partial \D)$. This will allow us to prove \eqref{eq:seconddisk} and \eqref{eq:thirddisk}. These new points correspond to the red points in Figure \ref{fig: 3 points on circle}. It follows by the arguments in the proof of \eqref{eq:firstdisk} that $t_1:=2-\lambda \in \partial F_{k,1} (\partial \D)$. Moreover, the circle $A_1:=\partial B(0, 2-\lambda)$ is tangent to the "exterior" tangent circles $\partial B(1+\lambda i, \lambda)$ and $\partial B(1-\lambda i, \lambda)$ (who are also tangent to each other). Since, there is no triple tangency, we can apply Descartes' Theorem in order to find the curvature of the corresponding Apollonian circle  (the second largest red circle inside $\D$ in Figure \ref{fig: Descartes' circles}). Recall that Descartes' Theorem states that the curvatures $k_1,k_2,k_3,k_4$ of  four circles with $6$ double tangencies and no triple tangency satisfy the equation 
\begin{equation}\label{eq: Decartes}
   ( k_1+k_2+k_3+k_4)^2 = 2 ( k^2_1+k^2_2+k^2_3+k^2_4).
\end{equation}
\begin{figure}[h]
  \hfill
    \begin{center}
\includegraphics[scale=.3]{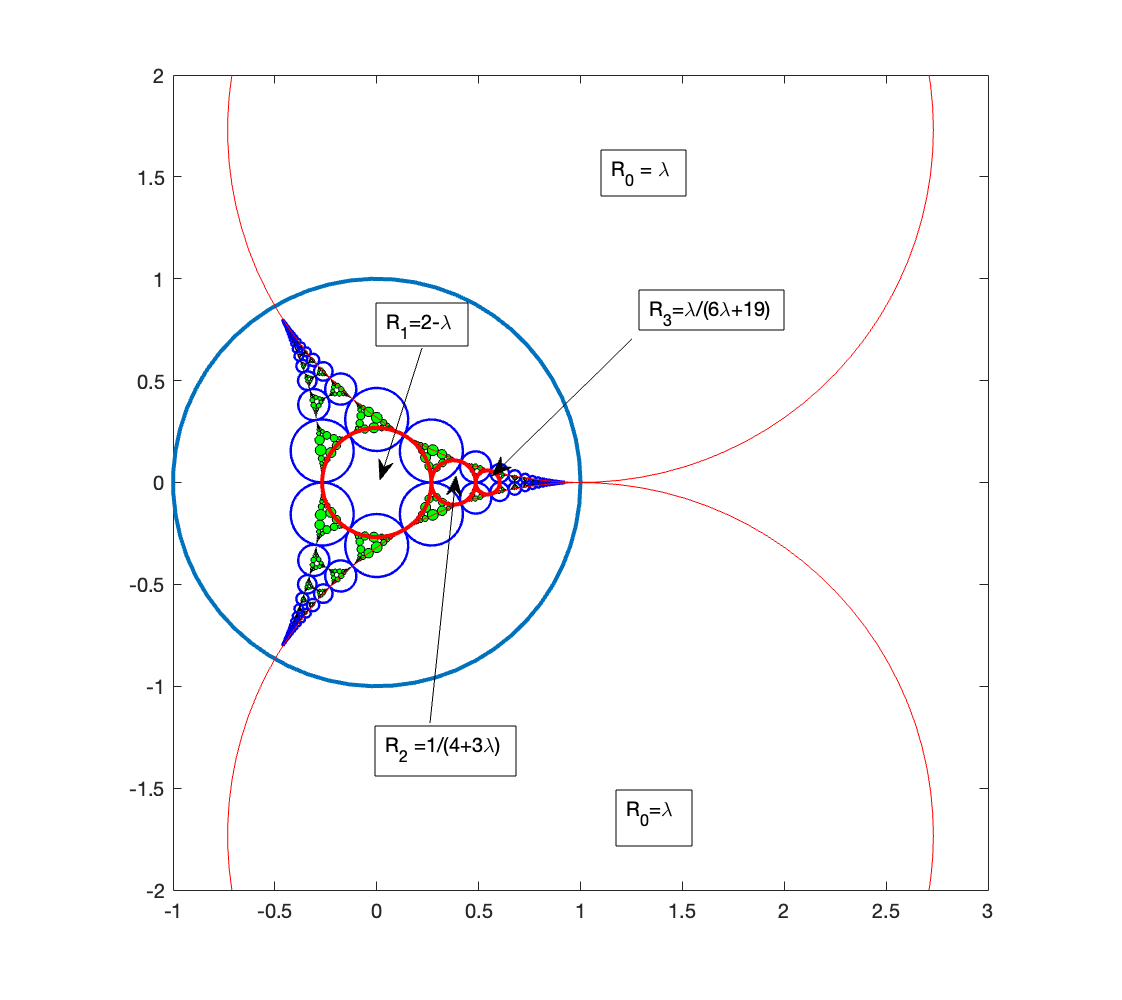}\\
   \caption{A sequence of Apollonian circles}
        \label{fig: Descartes' circles}
          \end{center}
  \hfill
\end{figure}
In our case
$$k_1=k_2=\frac{1}{\lambda} \mbox{ and }k_3=\frac{1}{2-\lambda}={2+\lambda}.$$
Thus the curvature and radius of the next Apollonian circle $A_2$ is
$$
k_4=4+3\lambda\quad \Rightarrow \quad R_2=\frac{1}{4+3\lambda}.
$$
Using this information we find that
$$t_2:=2+\lambda+ \frac{2}{4+3\lambda}=\frac{14-5\lambda}{11} \in F_{k,2}(\partial \D).$$

Thus, the following three distinct points
\begin{align*}
    t_2&= \frac{14-5\lambda}{11}\\
    b_1 &= \frac{2*2^2-1}{2*2 \left(2+\lambda \right)+2}+i\frac{\lambda}{2*2 \left(2+\lambda \right)+2}=\frac{7}{10+4\lambda }+i\frac{\lambda}{10+4\lambda }\\
     b_2 &= \frac{2*3^2-1}{2*3 \left(3+\lambda \right)+2}+i\frac{\lambda}{2*3 \left(3+\lambda \right)+2}=\frac{17}{20+6\lambda }+i\frac{\lambda}{20+6\lambda }.
\end{align*}
lie $F_{k,2}(\partial \D)$. Thus,
\begin{equation*}
\mbox{radius} F_{k,2}(\partial \D)= \frac{|t_2-b_1|\cdot|t_2-b_2|\cdot|b_1-b_2|}{4|\triangle (t_2,b_1,b_2)|}=\frac{14\lambda-15}{121}.
\end{equation*}
As in the proof of \eqref{eq:firstdisk} we see that that if $k$ is odd then $F_{k,2} (\D)=\overline F_{k+1, 2(\D)}$. Thus, \eqref{eq:seconddisk} follows because by \eqref{eq:rotationfkn}, the disks $\phi_{k,2} (\D)$ are obtained by rotations of $F_{1,2} (\D)$ and $F_{2,2} (\D)$.

We will use Descartes' formula \eqref{eq: Decartes} again to find the next Apollonian circle $A_3$ (the third largest red circle inside $\D$ in Figure \ref{fig: Descartes' circles}). In this case, we have $k_1=k_2=1/\lambda$ and $k_3=4+3\lambda$. Thus the curvature and radius of $A_3$ are 
$$
k_4=\frac{6\lambda+19}{\lambda}\quad  \mbox{ and }  R_3=\frac{\lambda}{6\lambda+19}.
$$
Using this information we find that
$$
t_3:=\frac{14-5\lambda}{11}+\frac{2\lambda}{6\lambda+19}=\frac{26-7\lambda}{23} \in F_{k,3} (\D).
$$
Thus, the following distinct points
\begin{align*}
    t_3&= \frac{26-7\lambda}{23}\\
    b_2 &= \frac{2*3^2-1}{2*3 \left(3+\lambda \right)+2}+i\frac{\lambda}{2*3 \left(3+\lambda \right)+2}=\frac{17}{20+6\lambda }+i\frac{\lambda}{20+6\lambda }\\
     b_3 &= \frac{2*4^2-1}{2*4 \left(4+\lambda \right)+2}+i\frac{\lambda}{2*4 \left(4+\lambda \right)+2}=\frac{31}{34+8\lambda }+i\frac{\lambda}{34+8\lambda }
\end{align*}
lie in $F_{k,3} (\D)$. Thus,
\begin{equation*}
\mbox{radius} F_{k,3}(\partial \D)= \frac{|t_3-b_2|\cdot|t_3-b_3|\cdot|b_2-b_3|}{4|\triangle (t_3,b_3,b_2)|}=\frac{26\lambda - 21}{529}.
\end{equation*}
Therefore, \eqref{eq:thirddisk} follows arguing exactly as (in the end of the) proof of \eqref{eq:seconddisk}. 
\end{proof}
\begin{remark}
\label{remk:radiidec}
If $k_m$ is the curvature of the Apollonian disk $A_m$ (the $m$ largest red disk inside $\D$ in Figure \ref{fig: Descartes' circles}) the curvature $k_{m+1}$ of the next Apollonian circle is given by
$$
k_{m+1}=k_m+\frac{2}{\lambda}+ 2\sqrt{\frac{2k_m}{\lambda}+\frac{1}{3}}.
$$
Thus, the radii of the Apollonian circles $A_m$ form a decreasing sequence $(R_n)_{n \in \N}$. One can then show that the sequence $\{\mbox{radius} \, \phi_{k,n}(\D)\}_{n \in \N}$ is also decreasing.
\end{remark}
The first step in bounding the distortion constant of $\mathcal{A}$ is obtaining an upper bound for the distortion of the maps $\f_{k,n}, (k,n) \in I$. 
\begin{prop}
\label{dist-dima}
For every $(k,n) \in I$:
$$
\frac{\max_{z \in \D} | \phi'_{k,n}(z)|}{\min_{z \in \D} | \phi'_{k,n}(z)|} \leq K_1,$$
where
$$K_1=\frac{\sqrt{\left(\frac{1}{2}+(1+\lambda)(\lambda+1)\right)^2+\frac{3}{4}}
+\sqrt{\left(\frac{3-\lambda}{2}+\lambda\right)^2+\frac{3}{4}\left(\lambda-1\right)^2}
}
{\sqrt{\left(\frac{1}{2}+(1+\lambda)(\lambda+1)\right)^2+\frac{3}{4}}
-\sqrt{\left(\frac{3-\lambda}{2}+\lambda\right)^2+\frac{3}{4}\left(\lambda-1\right)^2}}\approx 3.53765052763825.$$
\end{prop}
\begin{proof} 
We saw in the proof of Proposition \ref{prop:disks} that 
the map
\[\phi_{k,n}(z) = R_{\theta_k'} \circ f^n \circ R_{\theta_k} \circ f(z)\]
has matrix representation 
\begin{equation*}
\begin{split}
\Phi_{k,n} &:= n\lambda^{n-1} \begin{pmatrix}
e^{i \theta_k'} & 0 \\
0 & 1
\end{pmatrix} 
\begin{pmatrix}
-1 & 1 \\
-1 & 0
\end{pmatrix} 
\begin{pmatrix}
\lambda/n & 1 \\
0 & \lambda/n
\end{pmatrix}
\begin{pmatrix}
0 & -1 \\
1 & -1
\end{pmatrix} 
\begin{pmatrix}
e^{i\theta_k} &0 \\
0 & 1 
\end{pmatrix}
\begin{pmatrix}
\lambda -1 & 1 \\
-1 & \lambda +1
\end{pmatrix}  \\
&= n\lambda^{n-1} \begin{pmatrix}
e^{i \theta_k'} & 0 \\
0 & 1
\end{pmatrix} 
\begin{pmatrix}
-1+e^{i \theta_k}(\lambda-1)(-1+\frac{\lambda}{n}) & e^{i \theta_k}(-1+\frac{\lambda}{n})+(\lambda+1) \\
-e^{i \theta_k}(\lambda-1)-(1+\frac{\lambda}{n}) & -e^{i \theta_k}+(1+\frac{\lambda}{n})(\lambda+1)
\end{pmatrix}.
\end{split}
\end{equation*}
Recall that any M\"obius transformation  
\[g(z) = \frac{az+b}{cz+d}, \quad c \neq 0,\]
has a matrix representation
\[M_g = \begin{pmatrix}
a & b \\ c & d
\end{pmatrix},
\]
and  
\begin{equation}\label{eq:mobius-deriv}
  |g'(z)| = \frac{|\det(M_g)|}{|cz+d|^2}, \quad z\neq -d/c .  
\end{equation}
Therefore,
$$
K_{k,n}= \frac{\max_{z\in \D}|D \phi_{k,n}(z)|}{\min_{z\in \D} |D\phi_{k,n}(z)|}=
\frac{\max_{z\in \D}\frac{1}{|cz+d|^2}} {\min_{z\in \D}\frac{1}{|cz+d|^2}}=
\frac{\max_{z\in \D}\frac{1}{|z+d/c|^2}} {\min_{z\in \D}\frac{1}{|z+d/c|^2}},
$$
where 
$$
c = -e^{i \theta_k}(\lambda-1)-\left(1+\frac{\lambda}{n}\right)
\mbox { and }
d= -e^{i \theta_k}+\left(1+\frac{\lambda}{n}\right)(\lambda+1).
$$

From basic complex analysis we know that the minimum and maximum of $|z-A|$ on the circle $|z-z_0|=r$ are attained at 
$$
z=A+\left(1-\frac{r}{|z_0-A|}\right)(z_0-A),
$$
and 
$$
z=A+\left(1+\frac{r}{|z_0-A|}\right)(z_0-A),
$$
respectively, and the  minimum and maximum  are
$$
|z-A|=||z_0-A|-r|,
$$
and 
$$
|z-A|=|z_0-A|+r.
$$
In our situation we have   $A=-d/c$, $r=1$, and $z_0=0$, and as a result 
$$
K_{k,n}= \left(\frac{|d/c|+1^2}{|d/c|-1}\right)^2=\left(\frac{|d|+|c|}{|d|-|c|}\right)^2=\left(\frac{|-e^{i \theta_k}+(1+\frac{\lambda}{n})(\lambda+1)|+|-e^{i \theta_k}(\lambda-1)-(1+\frac{\lambda}{n})|}{|-e^{i \theta_k}+(1+\frac{\lambda}{n})(\lambda+1)|-|-e^{i \theta_k}(\lambda-1)-(1+\frac{\lambda}{n})|}\right)^2.
$$
Using that $e^{i \theta_k}=-\frac{1}{2}\pm\frac{\lambda}{2}i$ for $\theta_k=\pm \frac{2\pi}{3}$, we have:
\begin{equation*}
\begin{split}
&\frac{|-e^{i \theta_k}+(1+\frac{\lambda}{n})(\lambda+1)|+|-e^{i \theta_k}(\lambda-1)-(1+\frac{\lambda}{n})|}{|-e^{i \theta_k}+(1+\frac{\lambda}{n})(\lambda+1)|-|-e^{i \theta_k}(\lambda-1)-(1+\frac{\lambda}{n})|}\\
&\quad\quad\quad= \frac{\sqrt{\left(\frac{1}{2}+(1+\frac{\lambda}{n})(\lambda+1)\right)^2+\frac{3}{4}}
+\sqrt{\left(\frac{3-\lambda}{2}+\frac{\lambda}{n}\right)^2+\frac{3}{4}\left(\lambda-1\right)^2}
}
{\sqrt{\left(\frac{1}{2}+(1+\frac{\lambda}{n})(\lambda+1)\right)^2+\frac{3}{4}}
-\sqrt{\left(\frac{3-\lambda}{2}+\frac{\lambda}{n}\right)^2+\frac{3}{4}\left(\lambda-1\right)^2}}:=h(n).\end{split}
\end{equation*}
Thus,
$$
K_1=\sup_{(k,n) \in I} K_{k,n}=\sup_{n\in \N} (h(n))^2.
$$
It remains to show that 
\begin{equation}
\label{eq:bdbyh(1)}
h(n) \leq h(1), \quad \forall n \in \N.
\end{equation}
For $n\in [t_{min},\infty)$ with $t_{min}>0$, we have 
$$
0\le \frac{\lambda}{n}\le \frac{\lambda}{t_{min}}\quad\text{and}\quad
\left(\frac{3-\lambda}{2}\right)^2\le \left(\frac{3-\lambda}{2}+\frac{\lambda}{n}\right)^2\le \left(\frac{3-\lambda}{2}+\frac{\lambda}{t_{min}}\right)^2,
$$
and as a result for $n\in [t_{min},\infty)$
$$
h(n) \le  \frac{\sqrt{\left(\frac{1}{2}+(1+\frac{\lambda}{t_{min}})(\lambda+1)\right)^2+\frac{3}{4}}
+\sqrt{\left(\frac{3-\lambda}{2}+\frac{\lambda}{t_{min}}\right)^2+\frac{3}{4}\left(\lambda-1\right)^2}
}
{\sqrt{\left(\frac{1}{2}+(\lambda+1)\right)^2+\frac{3}{4}}
-\sqrt{\left(\frac{3-\lambda}{2}+\frac{\lambda}{t_{min}}\right)^2+\frac{3}{4}\left(\lambda-1\right)^2}}.
$$
Thus, for example for $t_{min}=23$, i.e. for $n\in[23,\infty]$, we have 
$$
h(n) \le  \frac{\sqrt{\left(\frac{1}{2}+(1+\frac{\lambda}{23})(\lambda+1)\right)^2+\frac{3}{4}}
+\sqrt{\left(\frac{3-\lambda}{2}+\frac{\lambda}{23}\right)^2+\frac{3}{4}\left(\lambda-1\right)^2}
}
{\sqrt{\left(\frac{1}{2}+(\lambda+1)\right)^2+\frac{3}{4}}
-\sqrt{\left(\frac{3-\lambda}{2}+\frac{\lambda}{23}\right)^2+\frac{3}{4}\left(\lambda-1\right)^2}}\approx 1.877,
$$
and checking directly for $n=1,2,\dots,22$, we have 

\begin{figure}[h]
  \hfill
   \begin{minipage}[t]{.4\textwidth}
    \begin{center}
\begin{verbatim}
h(1)=  1.880864303355842
h(2)=  1.806934748354880
h(3)=  1.776916017222780
h(4)=  1.761874740622607
h(5)=  1.753291620435355
h(6)=  1.747939742649673
h(7)=  1.744380723178319
h(8)=  1.741895325796543
h(9)=  1.740091692069607
h(10)= 1.738741615928681
h(11)= 1.737704919314939
h(12)= 1.736891653749831
h(13)= 1.736241960528771
h(14)= 1.735714756401264
h(15)= 1.735281087441614
h(16)= 1.734920073430289
h(17)= 1.734616350122985
h(18)= 1.734358409077515
h(19)= 1.734137492380911
h(20)= 1.733946840037007
h(21)= 1.733781167059785
h(22)= 1.733636293517403
\end{verbatim}
    \end{center}
  \end{minipage}
  \hfill
  \begin{minipage}[t]{.5\textwidth}
    \begin{center}
\includegraphics[scale=.2]{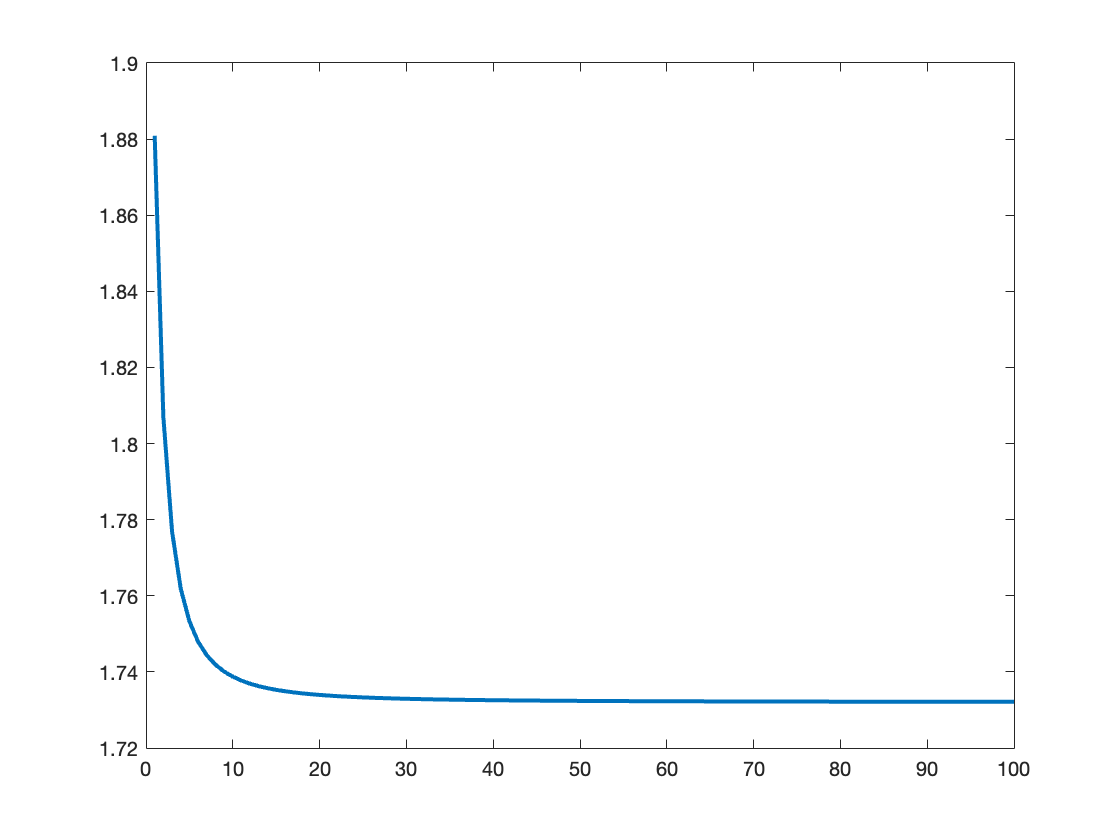}\\
  \caption{The graph of $h: \R^+ \to \R^+$}
        \label{fig: h(t)}
          \end{center}
  \end{minipage}
  \hfill
\end{figure}

Thus, we have established \eqref{eq:bdbyh(1)} and the proof is complete.
\end{proof}

\begin{remark} Interestingly, 
\begin{equation*}
\begin{split}
\lim_{n \to \infty} K_{k,n}&=\lim_{n \to \infty} h(n)^2 \\
&=\lim_{n\to\infty}\frac{\left(|-e^{i \theta_k}+(1+\frac{\lambda}{n})(\lambda+1)|+|-e^{i \theta_k}(\lambda-1)-(1+\frac{\lambda}{n})|\right)^2}{\left(|-e^{i \theta_k}+(1+\frac{\lambda}{n})(\lambda+1)|-|-e^{i \theta_k}(\lambda-1)-(1+\frac{\lambda}{n})|\right)^2}\\
&=
\left(\frac{|-e^{i \theta_k}+\lambda+1|+|-e^{i \theta_k}(\lambda-1)-1|}{|-e^{i \theta_k}+\lambda+1|-|-e^{i \theta_k}(\lambda-1)-1|}\right)^2=3,
\end{split}
\end{equation*}
where we used that $\theta_k=\pm \frac{2\pi}{3}$ and  $\lambda =\sqrt{3}$.
\end{remark}
We can now use Propositions \ref{prop:disks} and \ref{dist-dima} to obtain the desired distortion bounds.
\begin{prop}
\label{distprop}If $\om \in I^\ast$ then
\begin{equation}
\label{bdk}\frac{\max_{z \in \D} | \phi'_\om(z)|}{\min_{w \in \D} | \phi'_\om (w)|} \leq 5.900319:=K_{\mathcal{A}}.
\end{equation}
If $\om \in I_{n>1}^\ast$ then
\begin{equation}
\label{bdk>1}\frac{\max_{z \in \D} | \phi'_\om(z)|}{\min_{w \in \D} | \phi'_\om (w)|} \leq 5.03661:=K_{\mathcal{A}_{n>1}}.
\end{equation}
If $\om \in I_{n>2}^\ast$ then
\begin{equation}
\label{bdk>2}\frac{\max_{z \in \D} | \phi'_\om(z)|}{\min_{w \in \D} | \phi'_\om (w)|} \leq 4.3655:=K_{\mathcal{A}_{n>2}}.
\end{equation}
\end{prop}

\begin{proof} 
Let $z,w \in \D$. We will first prove \eqref{bdk}. Let $\om \in I^n$. If $n=1$, then \eqref{bdk} follows from Proposition \ref{dist-dima}. So we can assume that $|\om|>2$ and we set 
$g=\phi_{\om_1} \circ \dots \circ \phi_{\om_{n-1}}$.  Then, by Proposition \ref{dist-dima},
\begin{equation}
\label{eq:bdk-mar}
\frac{|\phi'_\om (z)|}{|\phi'_\om (w)|}=\frac{|(g \circ \phi_{\om_n})' (z)|}{|(g \circ \phi_{\om_n})' (w)|}=\left|\frac{g'( \phi_{\om_n}(z))}{g'( \phi_{\om_n}(w))}\right| \left| \frac{\phi'_{\om_n}(z)}{\phi'_{\om_n}(w)} \right| \leq K_{\mathcal{A}}\left|\frac{g'( \phi_{\om_n}(z))}{g'( \phi_{\om_n}(w)}\right|.
\end{equation}
We will now estimate 
$$\left|\frac{g'( \phi_{\om_n}(z))}{g'( \phi_{\om_n}(w))}\right|.$$
We consider two cases. First assume that $\om_1=(k,1)$ for some $k=1,\dots,6$. In that case, Proposition \ref{prop:disks}  implies that $\tilde{z}:=\phi_{\om_n}(z)$ and $\tilde{w}:=\phi_{\om_n}(w)$ lie in a disk $B(\xi,\frac{2 \lambda-3}{3})$ such that $\xi \in \partial B(0, \frac{4 \lambda-6}{3})$. It follows that $g$ is holomorphic (and injective) in $B(\xi,1+\lambda-|\xi|)$. By Koebe's Distortion Theorem, see e.g. \cite[Theorem 23.1.6]{MRUvol3}, (for $r=\frac{2 \sqrt{3}-3}{3}$ and $R=1+\lambda-\frac{4 \lambda-6}{3}$), we get
\begin{equation}
\label{eq:bdk-koebe}
\left|\frac{g'( \tilde{z})}{g'(\tilde{w})}\right|\leq \left(\frac{1+r/R}{1-r/R} \right)^4 \leq 1.6678634.
\end{equation}
Therefore,
\begin{equation}
\label{eq:bdkproof1}
\frac{|\phi'_\om (z)|}{|\phi'_\om (w)|} \overset{\eqref{eq:bdk-mar} \wedge \eqref{eq:bdk-koebe}}{\leq} 5.900319.
\end{equation}

If $\om_1=(k,m)$ for some $m>1$ and $k=1,\dots,6$ then Proposition \ref{prop:disks} and Remark \ref{remk:radiidec} imply that $\tilde{z},\tilde{w}$ lie in a disk $B(\zeta, \tilde{r})$ such that  $B(\zeta,\tilde{r}) \subset \D$ and 
$$
\tilde{r} \leq \frac{14\lambda-15}{121}.
$$
Thus we can apply Koebe's distortion theorem in the disk $B(\zeta, \lambda)$, since $B(\zeta, \lambda) \subset B(0,1+\lambda)$. Thus for $\tilde{R}=\lambda$
\begin{equation}
\label{eq:bdk-koebe1}
\left|\frac{g'( \tilde{z})}{g'(\tilde{w})}\right|\leq \left(\frac{1+\tilde{r}/\tilde{R}}{1-\tilde{r}/\tilde{R}} \right)^4 \leq 
1.42372,
\end{equation}
and
\begin{equation}
\label{eq:bdkproof2}
\frac{|\phi'_\om (z)|}{|\phi'_\om (w)|} \overset{\eqref{eq:bdk-mar} \wedge \eqref{eq:bdk-koebe1}}{\leq} 5.03661.
\end{equation}
Note that \eqref{bdk} follows by \eqref{eq:bdkproof1} and \eqref{eq:bdkproof2}. Moreover, \eqref{eq:bdkproof2} implies \eqref{bdk>1}.

It remains to show \eqref{bdk>2}. Note that in that case $\om_1=(k,m)$ for some $m>2$ and $k=1,\dots,6$. Therefore, Proposition \ref{prop:disks} and Remark \ref{remk:radiidec} imply that $\tilde{z},\tilde{w}$ lie in a disk $B(\eta, r')$ such that  $B(\eta,r') \subset \D$ and 
$$
r' \leq \frac{26\lambda - 21}{529}.
$$
Applying Koebe's distortion theorem in the disk $B(\eta, \lambda)$  (with $r'$ as above and $R'=\lambda$) gives
\begin{equation}
\label{eq:bdk-koebe2}
\left|\frac{g'( \tilde{z})}{g'(\tilde{w})}\right|\leq \left(\frac{1+r'/R'}{1-r'/R'} \right)^4 \leq 
1.234.
\end{equation}
Therefore
\begin{equation}
\label{eq:bdkproof3}
\frac{|\phi'_\om (z)|}{|\phi'_\om (w)|} \overset{\eqref{eq:bdk-mar} \wedge \eqref{eq:bdk-koebe1}}{\leq} 4.3655.
\end{equation}
We have thus established \eqref{bdk>2} and the proof of the proposition is complete.
\end{proof}

\section{The dimension spectrum of the Apollonian gasket}
\label{sec:proof}
Let $\mathcal{S}=\{\phi_i:X \to X\}_{i \in I}$ be a CIFS. If $F \subset I$ we will let $\cS_F=\{\phi_i: X \to X\}_{i \in F}$ and we will denote $J_{\cS_F}:=J_F$. We will also denote the topological pressure of the subsystem $\cS_F$ by $P_{\cS_F}:=P_F$. Recall that the dimension spectrum of $\cS$  is defined as 
$$DS(\cS)=\{\dim_H(J_F):F \subset I\}.$$
We now state some facts from \cite{dimspec} that will be essential in our approach. We say that a well ordering $\prec$ on the alphabet $I$ is {\em natural} (with respect to $\cS$) if
$$a \prec b \implies \| \f_a'\|_\infty \geq \| \f_b'\|_\infty.$$
Moreover, if $I=\{i_n\}_{n \in \N}$ is an enumeration respecting $\prec$ we will let $$I(m)=\{i_1,\dots,i_m\}.$$
The following proposition will be used repeatedly in our proof.
\begin{prop} \cite[Proposition 6.17]{dimspec}
\label{keypropoint}
Let $\cS=\{\f_i: X \to X\}_{i \in I}$ be a CIFS. Let $K$ as in \eqref{bdp} and let $I=\{e_n\}_{n \in \N}$ be an enumeration of $I$ according to a natural order $\prec$. Suppose that there exist $t_1 \leq  t_2 \leq h(\cS)$ and $d \in \N$ such that 
\begin{enumerate}
\item $P_{I(d)}(t_1) \leq 0$,
\item $\sum_{n \geq k+1} \| \f_{i_n}'\|_\infty^{t_2} \geq K^{2t_2} \|\f_{i_k}'\|^{t_2}_\infty$ for all $k \geq d+1$.
\end{enumerate}
Then $[t_1,t_2] \subset DS(\cS)$.
\end{prop}

We now turn our attention to the Apollonian system $\mathcal{A}$; see \eqref{eq:apolsystem}. Recall that the alphabet of $\mathcal{A}$ is
$$I=\{1,\dots,6\} \times \N.$$ We define the following well ordering of $I$:
\begin{equation}
\label{eq:prec}
(k_1,n_1) \prec (k_2, n_2) \mbox{ if and only if }  \begin{cases} n_1<n_2, \,\mbox{ or,}\\
n_1=n_2 \mbox{ and }k_1 \leq k_2\end{cases}.
\end{equation}
\begin{prop} The well ordering defined in \eqref{eq:prec} is natural, i.e. 
$$(k_1,n_1) \prec (k_2, n_2) \implies \|\phi'_{k_1,n_1}\|_\infty \geq \|\phi'_{k_2,n_2}\|_\infty.$$
\end{prop}
\begin{proof} In \cite{chousionis2024rigoroushausdorffdimensionestimates},  we showed that for all $k=1,\dots,6,$ and $n \in \N$
\begin{equation}\label{eq: estimate for Dphi}
\|\phi'_{k,n}\|_\infty= \frac{\lambda^2}{n^2}\frac{(2+\lambda)^2}{\left(|-1+\lambda i-(1+\lambda/n)(2+\lambda)|-\lambda\right)^2}:=\lambda^2 (2+\lambda)^2 \, G(n).
\end{equation}
Therefore, it suffices to show that $G: \N \to \N$ is decreasing. We first compute the derivative of $G$:

$$
\begin{aligned}
\frac{d G}{dn}&=\frac{2 \lambda\,  \left(2+\lambda\, \right) \left(\left(2+\lambda\, \right) \left(\frac{\lambda}{n}+1\right)+1\right)}{\left(\sqrt{\left(\left(2+\lambda\,
\right) \left(\frac{\lambda}{n}+1\right)+1\right)^2+3}\, -\lambda\, \right)^3 \sqrt{\left(\left(2+\lambda\, \right) \left(\frac{\lambda}{n}+1\right)+1\right)^2+3}\,
 n^4}\\
  &\quad\quad\quad\quad\quad-\frac{2}{\left(\sqrt{\left(\left(2+\lambda\, \right) \left(\frac{\lambda}{n}+1\right)+1\right)^2+3}\, -\lambda\, \right)^2 n^3}.
\end{aligned}
$$
To show that the derivative is negative, it is sufficient to show that 
$$
H(n):=\frac{ \lambda\,  \left(2+\lambda\, \right) \left(\left(2+\lambda\, \right) \left(\frac{\lambda}{n}+1\right)+1\right)}{n\left(\sqrt{\left(\left(2+\lambda\,
\right) \left(\frac{\lambda}{n}+1\right)+1\right)^2+3}\, -\lambda\, \right) \sqrt{\left(\left(2+\lambda\, \right) \left(\frac{\lambda}{n}+1\right)+1\right)^2+3}\,
 }<1
$$
for all $n\geq 1$.
Using that for $n_{min}\le n<\infty$ the estimate
$$
H(n)\le \frac{ \lambda\,  \left(2+\lambda\, \right) \left(\left(2+\lambda\, \right) \left(\frac{\lambda}{n_{min}}+1\right)+1\right)}{n_{min}\left(\sqrt{\left(\left(2+\lambda\,
\right)+1\right)^2+3}\, -\lambda\, \right) \sqrt{\left(\left(2+\lambda\, \right) +1\right)^2+3}\,
 }<1
$$
already holds for $n_{min}=3$, by checking directly for $n=1,2$, we find  
\begin{verbatim}
H(1) =  0.665616704035170
H(2) =  0.492071634601904.
\end{verbatim}
The proof is complete.
\end{proof}

Since 
$$
|-1+\lambda i-(2+\lambda)|\le |-1+\lambda i-(1+\lambda/n)(2+\lambda)|\le |-1+\lambda i-(1+\lambda)(2+\lambda)|,
$$
\eqref{eq: estimate for Dphi} implies that for any  $(k,n) \in I$:
\begin{equation}
\label{eq:derbound}
\frac{0.45}{n^2} < \|\phi'_{k,n}\|_\infty< \frac{3.821}{n^2}.
\end{equation}
Therefore, \eqref{eq:thetaz1} implies that
\begin{equation}
\label{eq:thetaapol}
\theta({\mathcal{A}})=\frac{1}{2}.
\end{equation}
\begin{theorem}
\label{mainthm}The Apollonian system $\mathcal{A}$ has full dimension spectrum; i.e.
$$DS(\mathcal{A})=[0, \dim_H(J_{\mathcal{A}})].$$
\end{theorem}
\begin{proof} We first note that since $\theta(\mathcal{A})=1/2$, \cite[Theorem 6.2]{MUCIFS} implies that
\begin{equation}
\label{eq:muthetspe}
[0,1/2) \subset DS(\mathcal{A}).
\end{equation}

We will use several subsystems of $\mathcal{A}$ of the form $\mathcal{A}_S$ where $S \subset \N$; recall \eqref{eq:subsystemsofA} for their definition. We let $E=\{e_n\}_{n \in \N}$ be an enumeration of the alphabet $I=\{(k,n):k=1,\dots,6, \quad n \in\N\}$ with respect to the natural order $\prec$ defined in \eqref{eq:prec}. Since, by \eqref{eq: estimate for Dphi}, $\|\phi'_{k,n}\|$ does not depend on $k$ we will let $\phi_n:=\phi_{k,n}, k=1,\dots, 6$. Note that
\begin{equation}
\label{eq:fracpartder}
\|\phi'_{e_m}\|_\infty=\|\phi'_{\left\lceil m/6 \right\rceil}\|_\infty
\end{equation}
were $\left\lceil x \right\rceil$ denotes the least integer  greater than or equal to $x$. 
We will repeatedly use the following corollary of Proposition \ref{keypropoint}.
\begin{cor}
\label{usefulcoro}Let $F \subset \N$ and let $\tilde{F}$ be a finite subset of $F$. Let $K_F$ be a distortion constant for $\mathcal{A}_{F}$. If $0<t_1<t_2$ and
\begin{enumerate}
\item $\dim_H (J_{\mathcal{A}_F})\leq t_1$,
\item $6 \sum_{n=M+1}^\infty \|\phi'_n\|^{t_2}_{\infty} \geq K^{2t_2}_{F} \|\phi'_M\|^{t_2}_\infty \mbox{ for all }M \geq \max \tilde{F}+1,$
\end{enumerate}
then $[t_1,t_2] \subset DS(\mathcal{A_{\tilde{F}}})\subset DS(\mathcal{A})$.
\end{cor}
\begin{proof} We let
$$I_{F}=\{(k,n): k=1,\dots, 6, \quad n \in F\}.$$
Let $(\epsilon_i)_{i \in \N}$ be an  enumeration of $I_F$ with respect to the natural order $\prec$, as in \eqref{eq:prec}.
Clearly, $\{1,\dots, 6\} \times \tilde{F}$ corresponds to the initial segment $I(q)$ of $\tilde{I}$ where $q= 6 \cdot \sharp F$. Therefore, our claim will follow by Proposition \ref{keypropoint} if we show that
$$\sum_{n=d+1}^\infty \|\phi'_{\epsilon_n}\|^{t_2}_\infty \geq K_{F}^{2t_2} \|\phi'_{\epsilon_d}\|^{t_2}_\infty \mbox{ for all }d \geq q+1.$$
Let $d \geq q+1$. Note that $\left\lceil d/6 \right\rceil \geq \max \tilde{F}+1$. Therefore,
$$\sum_{n=d+1}^\infty \|\phi'_{\epsilon_n}\|^{t_2}_\infty \geq 6 \sum_{m=\left\lceil d/6 \right\rceil+1}^\infty \|\phi'_{m}\|^{t_2}_\infty \geq K_{F}^{2t_2} \|\phi'_{\left\lceil d/6 \right\rceil}\|^{t_2}_\infty\overset{\eqref{eq:fracpartder}}{=}K_{F}^{2t_2} \|\phi'_{e_d}\|^{t_2}_\infty.$$
\end{proof}
By the integral test
\begin{equation}
\begin{split}
\label{eq:inttest}
\sum_{k=1}^6\sum_{n=N+1}^\infty \|\phi'_{k,n}\|_\infty^t &\overset{\eqref{eq: estimate for Dphi}}{=} 6\sum_{n=N+1}^\infty \|\phi'_{n}\|_\infty^t\overset{\eqref{eq:derbound}}{\geq} \sum_{n=N+1}^\infty\frac{6\,(0.45)^t}{n^{2t}}\\
&\geq 6*(0.45)^t\int_{N+2}^\infty\frac{dx}{x^{2t}}= \frac{6\,(0.45)^t}{(2t-1)(N+1)^{2t-1}}.
\end{split}
\end{equation}
By \eqref{eq:inttest} and the fact that the function $x \to \frac{x^{2t}}{(x+1)^{2t-1}}$ is increasing for $x>0$ we deduce that
\begin{equation}
\label{eq:simplifiedsumcond}
\frac{6\,(0.45)^t}{(2t-1)(N+1)^{2t-1}}\geq K^{2t}\frac{(3.821)^t}{N^{2t}}
\end{equation}
implies 
\begin{equation}
\label{eq: condition on N}
\sum_{n=M+1}^\infty \|\phi'_n\|^t_\infty \geq K^{2t} \|\phi'_M\|^t_\infty \mbox{ for all }M \geq N.
\end{equation}
Moreover, since the function
$t \to \frac{6\, (N+1)}{2t-1} \left(\left(\frac{N}{N+1}\right)^2 \frac{0.45}{3.821 \, K^2 } \right)^t$ 
is decreasing for $t>1/2$, if $t>s>1/2$
\begin{equation}
\label{eq:decrt}
\frac{6\,(0.45)^t}{(2t-1)(N+1)^{2t-1}}\geq K^{2t}\frac{(3.821)^t}{N^{2t}} \implies \frac{6\,(0.45)^s}{(2s-1)(N+1)^{2s-1}}\geq K^{2s}\frac{(3.821)^s}{N^{2s}}.
\end{equation}

We will now use Corollary \ref{usefulcoro} in order to find a chain of intervals in $DS(\mathcal{A})$ whose union contains $[1/2, \dim_{H}(J_{\mathcal{A}})]$. Starting with an interval which contains $\dim_{H}(J_{\mathcal{A}})$ we will use a bootstrapping method in order to reach $1/2$ with overlapping intervals in $DS(\mathcal{A})$. Our computational method from \cite{chousionis2024rigoroushausdorffdimensionestimates} is absolutely essential in this approach and will be applied repeatedly to derive lower and upper bounds for the Hausdorff dimension of various subsystems of $\mathcal{A}$, see Table \ref{tab:updimest}. For steps 1-16 we will use the distortion constant $K=K_{\mathcal{A}}=5.900319$ as in \eqref{bdk}. For steps 17 and 18  we will work with the dimension spectrum $\mathcal{A}_{n>2}$, so we will use the improved distortion constant  $$K_2:=K_{\mathcal{A}_{n >2}}=4.3655,$$
as in \eqref{bdk>2}. We describe in detail the first $4$ steps.

\textbf{Step 1.} We know that
$$\dim_H(J_{\mathcal{A}})<1.3057:=t_2.$$ For $t_2=1.3057$, \eqref{eq:simplifiedsumcond} holds for $N=454$. We then check that
$$
K^{2t_2}\|\phi'_{M}\|^{t_2}_\infty \le 6\sum_{n=M+1}^{10^7} \|\phi'_{n}\|_\infty^{t_2}$$
holds for $N \geq 27$. Thus, \eqref{eq: condition on N} holds for $t_2=1.3057$ and $N=27$. Applying the method we developed in \cite{chousionis2024rigoroushausdorffdimensionestimates} we obtain $$
\dim_H(\mathcal{A}_{n\le 26})\leq  1.3001:=t_1.
$$
Therefore, Corollary \ref{usefulcoro} implies that
\begin{equation}
\label{eq:step1}
[1.3001,\operatorname{dim}_H(\mathcal{A})]\subset {DS}(\mathcal{A}).
\end{equation}

{\bf Step 2.} Arguing as in Step 1, and using  \eqref{eq:decrt}, we find that condition \eqref{eq: condition on N} holds for $t_2=1.3001$ and $N=26$. We then find that 
$$
\dim_H(\mathcal{A}_{n\le 25})\leq  1.3000.
$$
Therefore, Corollary \ref{usefulcoro} implies that
$$
[1.3,1.3001]\subset {DS}(\mathcal{A}),
$$
which combined with \eqref{eq:step1} implies that
\begin{equation}
\label{eq:step2}
[1.300,\operatorname{dim}_H(\mathcal{A})]\subset {DS}(\mathcal{A}).
\end{equation}

{\bf Step 3.} In this step we consider the subsystem $\mathcal{A}_{n \neq 11,12}$, i.e. the Apollonian subsystem with alphabet $\{(k,n): k=1,\dots, 6, \, n\in \N \setminus \{ 11,12\}\}$, whose Hausdorff dimension satisfies:  
$$
\operatorname{dim}_\mathcal{H}(\mathcal{A}_{n\neq 11,12})\geq  1.3.
$$
Arguing as in Step 1, and using  \eqref{eq:decrt}, we find that condition \eqref{eq: condition on N} holds for $t_2=1.300$ and $N=26$. Moreover,
$$
\operatorname{dim}_\mathcal{H}(\mathcal{A}_{F_3})\leq  1.295,
$$
where $F_3=\{n\in \N: n\leq 25, n \neq 11,12\}$. Therefore, Corollary \ref{usefulcoro} implies that
$$
[1.295,1.3]\subset {DS}(\mathcal{A}),
$$
which combined with \eqref{eq:step2} implies that
\begin{equation}
\label{eq:step3}
[1.295,\operatorname{dim}_H(\mathcal{A})]\subset {DS}(\mathcal{A}).
\end{equation}

{\bf Step 4.} In this step we consider the subsystem $\mathcal{A}_{n \neq 11,12,13,14,15}$ whose dimension satisfies: 
$$
\operatorname{dim}_\mathcal{H}(\mathcal{A}_{n\neq \{11,12,13,14,15\}})\geq  1.296.
$$
Arguing as in Step 1, and using  \eqref{eq:decrt}, we find that condition \eqref{eq: condition on N} holds for $t_2=1.295$ and $N=26$. 
Moreover,
$$
\operatorname{dim}_\mathcal{H}(\mathcal{A}_{F_4})\geq  1.2901,
$$
where $F_4=\{n=1,\dots,10,16,\dots,25\}$. Therefore, Corollary \ref{usefulcoro} implies that
$$
[1.295,1.3]\subset {DS}(\mathcal{A}),
$$
which combined with \eqref{eq:step3} implies that
\begin{equation}
\label{eq:step4}
[1.2901,\operatorname{dim}_H(\mathcal{A})]\subset {DS}(\mathcal{A}).
\end{equation}
We continue in the same way until we obtain an interval which contains $1/2$. The details of the remaining steps (5--18) can be found in Table \ref{tab:main}.
\begin{table}[h]
\begin{center}
\begin{tabular}{|c|c|c|c|c|c|c|}
\hline
step &  $F$ & $D(F)$ & $\tilde{F}$ & $t_1(\geq \dim_{H}(J_{\mathcal{A}_{\tilde{F}}}))$ & $t_2$ & interval in $DS(\mathcal{A})$ \\
\hline 
1 &  $\N$  &1.3057&$\{n\le 26\}$  & 1.3001  & 1.3057 & $[1.3001, \operatorname{dim}_H(\mathcal{A})]$ \\
2 &  $\N$  &1.3057&  $\{n\le 25\}$  & 1.3000  &1.3001 & $[1.3000, \operatorname{dim}_H(\mathcal{A})]$ \\
3 & $\N \setminus \{11,12\}$ &1.300 & $\{n \leq 25\} \setminus \{11,12\}$  & 1.2950 & 1.3001  & $[1.2950, \operatorname{dim}_H(\mathcal{A})]$ \\
4 & $\N \setminus \{11,\dots,15\}$ & 1.296& $\{n \leq 25\}\setminus \{11,\dots,15\}$  & 1.2901  & 1.2950 &$[1.2901, \operatorname{dim}_H(\mathcal{A})]$ \\
5 & $\N \setminus \{11,\dots,20\}$  &1.291 &$\{ n \leq 24\} \setminus \{11,\dots,20\}$  & 1.2850 & 1.2901 &  $[1.2850, \operatorname{dim}_H(\mathcal{A})]$ \\
6 & $\N \setminus \{7,\dots,10\}$ &1.285& $\{n \leq 23\} \setminus\{7,\dots,10\}$  & 1.2775  & 1.2850 &  $[1.2775, \operatorname{dim}_H(\mathcal{A})]$ \\
7 & $\N \setminus \{6,\dots, 9\}$&1.278& $\{n \leq 23\} \setminus\{6,\dots, 9\}$  & 1.2700  & 1.2775 &  $[1.2700, \operatorname{dim}_H(\mathcal{A})]$ \\
8 & $\N \setminus \{6,\dots, 11\}$ &1.271& $\{n \leq 22\} \setminus\{6,\dots, 11\}$  & 1.2618  & 1.2700 & $[1.2618, \operatorname{dim}_H(\mathcal{A})]$ \\
9 & $\N \setminus \{5,\dots, 9\}$ &1.262& $\{n \leq 21\} \setminus\{5,\dots, 9\}$  & 1.2508  & 1.2618 &  $[1.2508, \operatorname{dim}_H(\mathcal{A})]$ \\
10 & $\N \setminus \{5,\dots, 12\}$ & 1.251 & $\{n \leq 20\} \setminus\{5,\dots, 12\}$  & 1.2383  & 1.2508 &  $[1.2383, \operatorname{dim}_H(\mathcal{A})]$ \\
11 & $\N \setminus \{4,\dots, 8\}$ &1.240& $\{n \leq 18\} \setminus\{4,\dots, 8\}$  & 1.2240  & 1.2383 & $[1.2240, \operatorname{dim}_H(\mathcal{A})]$ \\
12 & $\N \setminus \{3,4,5\}$ &1.2248&  $\{n \leq 17\} \setminus\{3,4,5\}$  & 1.2053  & 1.2240 & $[1.2035, \operatorname{dim}_H(\mathcal{A})]$ \\
13 & $\N \setminus \{3,4,5,6\}$  &1.211& $\{n \leq 15\} \setminus\{3,4,5,6\}$  & 1.1851  &1.2053  & $[1.185, \operatorname{dim}_H(\mathcal{A})]$ \\
14 & $\N \setminus \{2,3\}$ &1.186& $\{n \leq 14\} \setminus\{2,3\}$  & 1.1560  & 1.1851  & $[1.156, \operatorname{dim}_H(\mathcal{A})]$ \\
15 & $\N \setminus \{2,3,4\}$&1.1561& $\{n \leq 12\} \setminus\{2,3,4\}$  &  1.1080 &   1.1560 & $[1.108, \operatorname{dim}_H(\mathcal{A})]$ \\
16 & $\N \setminus \{1,5\}$  &1.110& $\{n \leq 9\} \setminus\{1,5\}$  &  1.0360 & 1.1080  & $[1.0360, \operatorname{dim}_H(\mathcal{A})]$ \\
$17^{\ast}$& $\N \setminus \{1,2\}$ &1.049& $\{n=3,4,5\}$  &  0.8261 & 1.0360 & $[0.8261, \operatorname{dim}_H(\mathcal{A})]$ \\
$18^{\ast}$& $\N \setminus \{1,2,3,4\}$ &0.964& $\{n=5\}$  & 0.4590  & 0.8261 & $[0.4590, \operatorname{dim}_H(\mathcal{A})]$ \\
\hline
\end{tabular}
\end{center}
\medskip 

\caption{\label{tab:main}The table summarizes how Corollary \ref{usefulcoro} is applied in each step of the iteration ($t_1,t_2, F$ and $\tilde{F}$ are as in that Corollary). In particular $t_1$ is an upper bound of $\dim_{H} (J_{\mathcal{A}_{\tilde{F}}})$ and $D(F)$ is a lower bound of $\dim_{H}(J_{\mathcal{A}_F})$. To be able to apply Corollary \ref{usefulcoro} we need to know that $t_2 \leq D(F)$. In steps 1 to 16, $K_F=K_{\mathcal{A}} =5.900319$. In steps 17 and 18, $K_F=K_{\mathcal{A}_{n>2}}=4.3655$, see Proposition \ref{distprop}.}
\end{table}
 \end{proof}
\begin{remark}
    The estimate for the size of the best distortion constant $K_\mathcal{A}$ plays a crucial role in our argument. A better estimate can lead to a reduced number of steps. However, the estimate we obtained in Proposition \ref{distprop} was sufficient for obtaining the first crucial step. One also may notice that the iterative argument makes rather small progress initially and then speeds up. We have also not tried to optimize the number of steps. There are many other choices for subsystems and it is likely that a different choice of subsystems would lead to the result faster.
\end{remark}

\begin{table}[!htb]
\begin{minipage}[t]{0.48\textwidth}
\begin{center}
\begin{tabular}{|c|c|}
\hline
subsystem of $\mathcal{A}$ & upper dim. estimate \\
\hline 
$\mathcal{A}_{n\leq 26}$      &  $1.3001$\\
$\mathcal{A}_{n\leq 25}$      &  $1.3000$\\
$\mathcal{A}_{\{n\leq 25\} \setminus \{11,12\}}$      &  $1.2950$\\
$\mathcal{A}_{\{n \leq 25\}\setminus \{11,\dots,15\}}$      &  $1.2901$\\
$\mathcal{A}_{\{n \leq 24\}\setminus \{11,\dots,20\}}$      &  $1.2850$\\
$\mathcal{A}_{\{n \leq 23\}\setminus \{7,\dots,10\}}$      &  $1.2775$\\
$\mathcal{A}_{\{n \leq 23\}\setminus \{6,\dots,9\}}$      &  $1.2700$\\
$\mathcal{A}_{\{n \leq 22\}\setminus \{6,\dots,11\}}$      &  $1.2618$\\
$\mathcal{A}_{\{n \leq 21\}\setminus \{5,\dots,9\}}$      &  $1.2508$\\
$\mathcal{A}_{\{n \leq 20\}\setminus \{5,\dots,12\}}$      &  $1.2383$\\
$\mathcal{A}_{\{n \leq 18\}\setminus \{4,\dots,8\}}$      &  $1.2240$\\
$\mathcal{A}_{\{n \leq 17\}\setminus \{3,4,5\}}$      &  $1.2053$\\
$\mathcal{A}_{\{n \leq 15\}\setminus \{3,4,5,6\}}$      &  $1.1851$\\
$\mathcal{A}_{\{n \leq 14\}\setminus \{2,3\}}$      &  $1.1560$\\
$\mathcal{A}_{\{n \leq 12\}\setminus \{2,3,4\}}$      &  $1.1080$\\
$\mathcal{A}_{\{n \leq 9\}\setminus \{1,5\}}$      &  $1.0360$\\
$\mathcal{A}_{\{n=3,4,5\}}$      &  $0.8261$\\
$\mathcal{A}_{\{n=5\}}$      &  $0.4581$\\
\hline
\end{tabular}
\end{center}
\end{minipage}
\hfill
\begin{tabular}
{|c|c|}
\hline
subsystem of $\mathcal{A}$ & lower dim. estimate \\
\hline 
$\mathcal{A}_{n\neq \{1,2,3,4\}}$      &  $0.964$\\
$\mathcal{A}_{n\neq \{1,2\}}$      &  $1.049$\\
$\mathcal{A}_{n\neq \{1,5\}}$      &  $1.110$\\
$\mathcal{A}_{n\neq \{2,3,4\}}$      &  $1.1561$\\
$\mathcal{A}_{n\neq \{2,3\}}$      &  $1.186$\\
$\mathcal{A}_{n\neq \{3,4,5,6\}}$      &  $1.211$\\
$\mathcal{A}_{n\neq \{3,4,5\}}$      &  $1.2248$\\
$\mathcal{A}_{n\neq \{4,5,6,7,8\}}$      &  $1.240$\\
$\mathcal{A}_{n\neq \{5,6,7,8,9,10,11,12\}}$      &  $1.251$\\
$\mathcal{A}_{n\neq \{4,5,6\}}$      &  $1.256$\\
$\mathcal{A}_{n\neq \{5,6,7,8,9\}}$      &  $1.262$\\
$\mathcal{A}_{n\neq \{6,7,8,9,10,11\}}$      &  $1.271$\\
$\mathcal{A}_{n\neq \{6,7,8,9,10\}}$      &  $1.274$\\
$\mathcal{A}_{n\neq \{6,7,8,9\}}$      &  $1.278$\\
$\mathcal{A}_{n\neq \{7,8,9,10\}}$      &  $1.285$\\
$\mathcal{A}_{n\neq \{11,\dots,20\}}$      &  $1.291$\\
$\mathcal{A}_{n\neq \{11,12,13,14,15\}}$      &  $1.296$\\
$\mathcal{A}_{n\neq \{11,12\}}$      &  $1.300$\\
\hline
\end{tabular} 
\medskip
\caption{\label{tab:updimest} Upper and lower bounds for the Hausdorff dimension of the subsystems of $\mathcal{A}$ used in the proof of Theorem \ref{mainthm}. }
\end{table}


\bibliography{biblio}
\bibliographystyle{siam}

\end{document}